\documentclass[12pt]{article}
\usepackage{amsthm}
\usepackage{amssymb,latexsym,amsmath}

\oddsidemargin=\evensidemargin \footskip=36pt \voffset=-1.5cm
\DeclareMathOperator{\Mat}{Mat}

\DeclareMathOperator{\Irr}{Irr}
\def\tm#1{\item[{\rm (#1)}]}
\def\css{\begin{cases}}
\def\ecss{\end{cases}}

\def\C{{\mathbb C}}

\def\CC{{\cal C}}
\def\S{\mathcal{S}}

\def\h{{\mathcal{H} }}

\def\bull{\vrule height 1.2ex width 1ex depth -.1ex }

\makeatletter
\renewcommand{\subsection}{\@startsection{subsection}{2}{0mm}{-2mm}{-2mm}
{\bf\normalsize}}

\def\nmrt{\begin{enumerate}}
\def\enmrt{\end{enumerate}}
\def\tm#1{\item[{\rm (#1)}]}

\makeatother
\newtheorem{formula}{}[section]

\newtheorem{definition}[formula]{Definition}
\newtheorem{corollary}[formula]{Corollary}
\newtheorem{remark}[formula]{Remark}
\newtheorem{lemma}[formula]{Lemma}
\newtheorem{theorem}[formula]{Theorem}

\newtheorem{example}[formula]{Example}

\begin{document}

\title{Standard Character Condition for C-algebras}
\author{
J. Bagherian and A. Rahnamai Barghi\footnote{Corresponding author: rahnama@iasbs.ac.ir} \\ \\
Institute for Advanced Studies in Basic Sciences \\
P.O. Box: 45195-1159, Zanjan-Iran \\ \\
{\tt  bagherian@iasbs.ac.ir}\\
{\tt  rahnama@iasbs.ac.ir}}

\maketitle
\begin{abstract}
\noindent
It is well known that the adjacency algebra of an
association scheme has the standard character. In this paper we
first define the concept of standard character for C-algebras and
we say that a C-algebra has the {\it standard character condition}
if it has the standard character. Then we investigate some
properties of C-algebras which have the standard character
condition and prove that under some conditions a C-algebra  has an
adjacency algebra homomorphic image. In particular, we obtain a
necessary and sufficient condition for which a commutative table algebra
comes from an association scheme.
\end{abstract}
\smallskip

\section{Introduction}
A table algebra is a C-algebra with nonnegative structure
constants was introduced by \cite{Ar2}. As a folklore example, the
adjacency algebra of an association scheme (or homogeneous
coherent configuration) is an integral table algebra. On the other hand,
the adjacency algebra of an association scheme has a special
character which is called the {\it standard character}, see \cite{HIG}.
We generalize the concept of standard character from adjacency
algebras to C-algebras. This generalization enables us to find a
necessary and sufficient condition for which a commutative table
algebra comes from an association scheme.

In section \ref{2309072} we recall the concept of C-algebras
and some related properties which we will use in this paper.

In section \ref{2309073} we first define the {\it standard
feasible trace} for C-algebras which is a generalization of the
standard character in the theory of association schemes.
Thereafter, we show that the standard feasible multiplicities of
the characters of a table algebra and its quotient are the same.
Furthermore, we prove that the set of standard feasible
multiplicities preserve under C-algebras isomorphism.

In section \ref{stchar} we give an example of C-algebra for which
the standard feasible trace is a character, such character is
called the {\it  standard character}. By using the standard
character we obtain a necessary and sufficient condition for which
a commutative table algebra comes from an association scheme.

\section{Preliminaries}\label{2309072}
Although in algebraic combinatorics the concept of C-algebra is
used for commutative algebras, in this paper we will also consider
non-commutative algebras.
Hence we deal with non-commutative C-algebras in the sense of \cite{EPV} as the following:\\

Let $A$ be a finite dimensional associative algebra over the
complex field $\C$ with the identity element $1_A$ and a base $B$
in the linear space sense. Then the pair $(A,B)$ is called a {\it
non-commutative C-algebra} if the following conditions (I)-(IV)
hold: \nmrt \tm{I} $1_A \in B$ and the structure constants of
${B}$ are real numbers, i.e., for $a,b \in {B}$:
$$
ab = \displaystyle\sum _{c\in {B}}\lambda_{abc}c,\ \ \ \ \
\lambda_{abc} \in \mathbb{R}.
$$
\tm{II} There is a semilinear involutory anti-automorphism
(denoted by $^{*}$) of $A$ such that ${B}^{*} = {B}$. \tm{III} For
$a,b\in B$ the equality $\lambda_{ab1_A} = \delta_{ab^*}|a|$ holds
where $|a|>0$ and $\delta$ is the Kronecker symbol. \tm{IV} The
mapping $b\rightarrow |b|, b\in B$ is a one dimensional
$*$-linear representation of the algebra $A$, which is called the
{\it degree map}. \enmrt

\begin{remark}
In the definition above we should mention that if the algebra $A$
is commutative, then $(A,B)$ becomes a C-algebra in the sense of
\cite{Ban}.
\end{remark}

If the structure constants of a given C-algebra (resp. commutative)
are nonnegative real numbers, then it is called a
{\it table algebra} (resp. {\it commutative}) in the sense of
\cite{Ar2} (resp. \cite{Ar1}).

Throughout this paper a C-algebra ( resp. table algebra) will
mean a non-commutative C-algebra (resp. non-commutative table algebra).

A C-algebra (table algebra) is called {\it integral} if all its structure
constants $\lambda_{abc}$ are integers.
The value $|b|$ is called the {\it degree} of the basis element
$b$. From condition (IV) we see that $|b| = |b^*|$ for all $b\in
B$, and from condition (II) for $a=\sum_{b\in B}x_bb$ we have
$a^*=\sum_{b\in B}\overline{x_b}b^*$, where $\overline{x_b}$ means
the complex conjugate to $x_b$. This implies that the Jacobson
radical $J(A)$ of the algebra $A$ is equal to $\{0\}$ which means
$A$ is semisimple.

Let $(A,B)$ and $(A',B')$ be two C-algebras. A {\it C-algebra
homomorphism} from $(A,B)$ to  $(A',B')$ is an $*$-algebra
homomorphism $f: A \to A'$ such that $f(B)=B'$. Such C-algebra
homomorphism is called {\it C-algebra epimorphism} (resp. {\it monomorphism})
if $f$ is onto (resp. into). A  C-algebra epimorphism $f$ is called
a {\it C-algebra isomorphism} if $f$ is monomorphism too. Two
C-algebras $(A,B)$ and $(A',B')$ are called
{\it isomorphic}, if there exists a C-algebra isomorphism between them.\\

Given a table algebra $(A,B)$, the bilinear form $\langle \cdot,
\cdot \rangle$ on $A$ is defined in \cite{Ar2} by setting $\langle
x, y \rangle =t(xy^*)$, for $x,y \in A$, where $t:A\rightarrow \C$
is a linear function defined by $t(\sum_{b\in B}x_bb)=x_{1_A}$. Then
one can see that $\langle \cdot, \cdot \rangle$ is a Hermitian
positively definite form on $A$.

A nonempty subset $C\subseteq B$ is called a {\it closed subset},
if ${C}^{*}{C} \subseteq {C}$. We denote by $\mathcal{C}(B)$ the
set of all closed subsets of $B$.

Let $(A,B)$ be a table algebra with the basis $B$ and let $C\in
\mathcal{C}(B)$. From \cite[Proposition 4.7]{Ar}, it follows that
$\{{C}b{C} \mid \ b\in {B}\}$ is a partition of ${B}$. A subset
${C}b{C}$ is called a {\it $C$-double coset} or {\it double coset}
with respect to the closed subset ${C}$. Let
$$
b/C := |C^+|^{-1}({C}b{C})^{+} = |C^+|^{-1}\displaystyle\sum_{x\in
{C}b C}x
$$
where $C^+ = \sum_{c\in C}c$ and  $|C^+|=\sum_{c\in C}|c|$. Then
the following theorem is an immediate consequence of \cite[Theorem
4.9]{Ar}:

\begin{theorem}
\label{Arad} Let $(A,{B})$ be a table algebra and let $C\in
\CC(B)$. Suppose that $\{b_1=1_A,\ldots ,b_k\}$ be a complete set
of representatives of $C$-double cosets. Then the vector space
spanned by the elements $b_i/{C}, 1\leq i\leq k$, is a table
algebra ( which is denoted by $A/{C}$) with a distinguished basis
${B}/{C} = \{b_i/{C} \mid \ 1\leq i \leq k \}.$ The structure
constants of this algebra are given by the following formula:
$$\gamma_{ijk}=
|C^+|^{-1}\sum_{r\in Cb_iC,s\in Cb_jC} \lambda_{rst} $$ where
$t\in C b_k C$ is an arbitrary element.\hfill \bull
\end{theorem}

The table algebra $(A/C,B/C)$ is called the {\it quotient table
algebra} of $(A,B)$ modulo $C$.

We refer the reader to \cite{Zi} for the background of association
schemes.

\section{The standard feasible trace for C-algebras}\label{2309073}
In this section we first define the standard feasible trace for
C-algebras and then we show that the standard feasible
multiplicities of the characters of a table algebra and
its quotient are the same. Furthermore, we prove that the set
of standard feasible multiplicities preserve under C-algebras isomorphism.\\

Let $(A,B)$ be a C-algebra and let $\Irr(A)$ be the set of
irreducible characters of $A$. We define a linear function $\zeta
\in \rm{Hom_{\C}(A,\C)}$ by $\zeta(b)=\delta_{1_Ab}|B^+|$, for
$b\in B$, where $|B^+| = \sum_{b\in B}|b|$. It is easily seen that
$\zeta(bc)=\zeta(cb)$, for all $b,c\in B$. This shows that $\zeta$
is a {\it feasible trace} in the sense of \cite{Hi}. In addition,
since ${\rm rad}\zeta = \{0\}$, where ${\rm rad}(\zeta) = \{ x\in
A: \zeta(xy) = 0,~ \forall y\in A \}$, it is a non-degenerate
feasible trace on $A$. Therefore, from \cite{Hi} it follows that
$\zeta=\displaystyle\sum _{\chi\in \rm{Irr(A)}}\zeta_\chi\chi$
where $\zeta_\chi\in \C$ and all $\zeta_\chi$ are non-zero. We
call $\zeta$ the {\it{standard feasible trace}}, $\zeta_\chi$
the {\it standard feasible multiplicity} of $\chi$ and $\{\zeta_\chi |~ \chi \in \Irr(A)\}$,
the {\it set of standard feasible multiplicities} of the C-algebra $(A,B)$.\\

Since $A$ is a semisimple algebra
$$
A=\displaystyle\bigoplus_{\chi\in \Irr(A)}A\varepsilon{_\chi}
$$
where $\varepsilon_{\chi}$'s are the central primitive
idempotents.
\begin{lemma}\label{130308}
\nmrt \tm{i} Let $\chi\in \Irr(A)$. Then
\begin{eqnarray}\label{23100701}
\varepsilon_{\chi}=\frac{1}{|B^+|}\displaystyle\sum _{b\in
B}\frac{\zeta_{\chi}\chi(b^*)}{|b^*|}b.
\end{eqnarray}
\tm{ii} {\rm{(Orthogonality Relation)}} For every $\phi, \psi \in
\rm{Irr(A)}$
\begin{eqnarray}\label{23100702}
\frac{1}{|B^+|}\displaystyle\sum_{b\in
B}\frac{1}{|b^*|}\phi(b^*)\psi(b)=
\delta_{\phi\psi}\frac{\phi(1)}{\zeta_{\phi}}.
\end{eqnarray}
\tm{iii}In commutative case, for  every $b,c \in B$
$$\displaystyle\sum_{\chi \in \Irr(A)}\zeta_\chi\chi(b)\chi(c^*)=\delta_{bc}|b||B^+|.$$
\enmrt
\end{lemma}

\begin{proof}
Part (i) and (iii) follow from \cite[5.7]{Hi} and
\cite[$5.5'$]{Hi}, respectively, by using the concept of dual basis
relative to a non-degenerate feasible trace, indeed the dual basis
of $b$ relative to standard feasible trace $\zeta$ is
$\frac{b^*}{|b||B^+|}$ for $b\in B$. Part (ii) follows from
equality
$\varepsilon_{\phi}\varepsilon_{\psi}=\delta_{\phi\psi}\varepsilon_{\phi}$
by replacing $b^*$ by $1_A$. \hfill\bull
\end{proof}
\begin{remark}\label{251007r}
From (\ref{23100701}) one can see that in commutative case,
$\zeta_{\chi}$ is the coefficient of $1_A$ in the linear
combination of $|B^+|\varepsilon_{\chi}$ in terms of the basis
elements of $B$.
\end{remark}

Let $(A,B)$ be a table algebra and $C\in \mathcal{C}(B)$. Set
$e=|C^+|^{-1}C^+$. Then $e$ is an idempotent for the table algebra
$A$ and the subalgebra $eAe$ denoted by $\h$, is equal to the
quotient table algebra $(A/C,B/C)$ modulo $C$, see \cite{Ar}. Let
$\zeta$ be the standard feasible trace of the table algebra
$(A,B)$. Then $\zeta|_{\h}$ is the standard feasible trace for
$(A/C,B/C)$. Indeed, assume that $T\subseteq B$ be a complete set
of representatives of $C$-double cosets of $A$. Then $B=\bigcup
_{b\in T} CbC$ and $|C^+|^{-1}|B^+|=\sum_{b\in
T}|b/C|$. Since
$$
\zeta|_{\h}(b/C)= \css
|C^+|^{-1}|B^+|,  &\text{if $b=1_A$}\\
0, &  \text{if $b\neq 1_A$}\\
\ecss
$$
it follows that $\zeta|_{\h}$ is the standard feasible trace for
$(A/C,B/C)$. Thus we proved the following lemma:

\begin{lemma}\label{280208}
Let $(A,B)$ be a table algebra with the standard feasible trace
$\zeta$ and let $C\in \mathcal{C}(B)$. Then $\zeta|_{\h}$ is the
standard feasible trace of $(A/C,B/C)$. \hfill\bull
\end{lemma}

In the following we will show that the standard feasible
multiplicities of the characters of a table algebra and its
quotient are the same. For this, we need to observe a relationship
between the characters of a table algebra and its quotient. The next
three theorems and corollaries are proved for adjacency
algebra of an association scheme, see \cite{HAN}. Now we generalize
them for table algebras.

\begin{theorem}\label{240108}
Let $(A,B)$ be a table algebra and let $P=\{\varepsilon_{\chi} |~
\chi \in \Irr(A) \}$ be the set of central primitive idempotents
of $(A,B)$. Then $P_C=\{e\varepsilon_{\chi}|~ \chi \in \Irr(A)
\}\setminus\{0\}$ is the set of central primitive idempotents of
the quotient table algebra $(A/C,B/C)$ where $C\in \mathcal{C}(B)$
and $e=|C^+|^{-1}C^+$.
\end{theorem}
\begin{proof}
Suppose that $\varepsilon_{\chi} \in P$ such that
$e\varepsilon_{\chi} \neq 0$. Then the algebra $\varepsilon_{\chi}
A$ is isomorphic to $\rm{End}_{A}(V)$ where $V=
\varepsilon_{\chi}A$. Let $T$ be the image of the idempotent
$e\varepsilon_{\chi}$ with respect to this isomorphism. From
\cite[Theorem 5.4]{Na}, it follows that the algebra
$e\varepsilon_{\chi}\h$ is isomorphic to $\rm{End}_{A}(TV)$. Since
the latter algebra is simple, so $e\varepsilon_\chi$ is a central
primitive idempotent of the algebra $\h$. On the other hand, since
$e$ is the unit element of $\h$ and $e=\sum e\varepsilon_\chi$
where $e\varepsilon_\chi$ runs over the set $
\{e\varepsilon_{\chi}|~ \chi \in \Irr(A) \}\setminus\{0\}$, we
conclude that $P_C$ is the set of all central primitive
idempotents of the quotient table algebra $(A/C,B/C)$ and we are done.\hfill\bull
\end{proof}
\begin{corollary}
\label{corresond} There is a one to one correspondence between
$\{\chi\in \Irr(A) |~ \chi|_\h \neq 0\}$ and $\rm{Irr}(\h)$ by the
map $\chi\rightarrow \chi|_\h $.
\end{corollary}
\begin{proof} This is an immediate consequence of Theorem
\ref{240108}. \hfill \bull
\end{proof}
\begin{corollary}\label{correspond2}
Let $(A,B)$ be a table algebra and $C\in\mathcal{C}(B)$. Then
every irreducible $A$-module $V$ is an irreducible $\h$-module
iff $\dim_\C(eV)\neq 0$, where $e=|C^+|^{-1}C^+$.
\end{corollary}
\begin{proof}
Let $V$ be an irreducible $A$-module and let $D$ be a matrix
representation of $A$ defined by $V$. Since $e$ is an idempotent,
rank $D(e)=\chi(e)$, where $\chi$ is the irreducible character afforded by
$D$. On the other hand, rank $D(e)=\dim_\C(eV)$. Hence
\begin{eqnarray}\label{correspond3}
\chi(e)=\dim_\C(eV).
\end{eqnarray}

We first suppose that $V$ is an irreducible $\h$-module. Then
$\chi|_\h\neq 0$, and so $\chi(e)\neq 0$. Thus equality
(\ref{correspond3}) implies that $\dim_\C(eV)\neq0$.

Conversely, let $\dim_\C(eV)\neq0$. From equality
(\ref{correspond3}) and Corollary \ref{corresond} we deduce that
$\chi|_\h$ is an irreducible character of $\h$. Thus $V$ is an
irreducible $\h$-module and we are done. \hfill\bull
\end{proof}
\vspace{.4cm} The theorem below gives the relationship between the
standard feasible multiplicity of a character of a table algebra $(A,B)$ and
the quotient table algebra $(A/C,B/C)$.

\begin{theorem}
The standard feasible multiplicity of $\chi|_{\h}$ is equal to
that of $\chi$ if  $\chi|_{\h} \neq 0$.
\end{theorem}
\begin{proof}
Let $\{\varepsilon_{\chi}~|~\chi \in \rm{Irr}(A)\}$ be the set of
central primitive idempotents of $A$. Then
$\zeta(e\varepsilon_\chi)=\zeta_{\chi}\chi(e\varepsilon_\chi)$. On
the other hand, from Theorem \ref{240108} we conclude that
$\zeta(e\varepsilon_\chi)=\zeta|_{\h}(e\varepsilon_\chi)$. But
from Lemma \ref{280208} it follows that
$\zeta|_{\h}(e\varepsilon_\chi)= \zeta_{\chi|_{\h}}{\chi|_{\h}}
(e\varepsilon_\chi)$, where $\zeta_{\chi|_{\h}}$ is the standard
feasible multiplicity of $\chi|_{\h}$.
 This implies that $\zeta_\chi\chi(e\varepsilon_\chi)=
\zeta_{\chi|_{\h}}{\chi|_{\h}}(e\varepsilon_\chi)$. Thus
$\zeta_\chi=\zeta_{\chi|_{\h}}$, as claimed.\hfill\bull
\end{proof}
\vspace{.4cm} Suppose that $(A,B)$ is a C-algebra and $\rho\in
\rm{Hom_\C(A,\C)}$ such that $\rho(b)=|b|$. Then $\rho$ is an
irreducible character of $A$, which is called the {\it principle
character} of $(A,B)$. From (\ref{23100702}) by replacing $\phi$
and $\psi$ by $\rho$ we conclude that $\zeta_{\rho}=1$. Moreover,
if $(A,B)$ is a commutative C-algebra, then \cite[Corollary
5.6]{Ban} shows that $\zeta_\chi > 0$. In the following lemma we
give a lower bound for the standard feasible multiplicities of
the characters of a table algebra.
\begin{lemma}
Let $(A,B)$ be a table algebra. Then $|\zeta_{\chi}| \geq
\chi(1_A)^{-1}$, for every $\chi\in \Irr(A)$. In particular, if
$(A,B)$ is commutative table algebra then $|\zeta_{\chi}| \geq 1$.
\end{lemma}
\begin{proof}
We first claim that $|\chi(a)|\leq |a|\chi(1)$, where $a\in B$ and
$\chi$ is a character of $A$. To do this, let $D$ be a
representation of $A$ which affords character $\chi$ and let $a\in
A$. Suppose that $m_a(x)$ is the minimal polynomial of $a$ and
Spec$(a)$ is the set of all roots of $m_a(x)$. Let $\lambda\in
{\rm{Spec}}(a)$. Then $a-\lambda.1_A$ can not be invertible, see
\cite[Corollary 2.25]{DF}. So there exists a non zero element
$x\in A$ such that $(a-\lambda.1_A)x=0$ or equivalently
$ax=\lambda x$. But by \cite[Proposition 2.3]{Ar2} we have $|\langle ax,x
\rangle|\leq |a|\langle x,x \rangle$ and so the latter equality
implies that $|\lambda \langle x,x \rangle |\leq |a|\langle
x,x\rangle $. Therefore $|\lambda|\leq |a|$. This fact along with
the obvious inclusion Spec$(D(a))\subseteq {\rm{Spec}}(a)$ prove
the claim. Now the result follows by applying the degree map
$|\cdot|$ on the both sides of the equation (\ref{23100702}).

The second statement is an immediate consequence of the first one,
since $\chi(1_A) = 1 $ for commutative case. \hfill\bull
\end{proof}

\begin{lemma}\label{251107}
The set of standard feasible multiplicities of two isomorphic
C-algebras are the same.
\end{lemma}
\begin{proof}
Let $(A,B)$ and $(A',B')$ be two C-algebras and $ f :(A,B) \to
(A',B')$ be an isomorphism. Let  $\zeta$ and $\zeta'$ be the
standard feasible traces of $(A,B)$ and $(A',B')$, respectively.
Let $P=\{\varepsilon_\chi\mid\chi \in \Irr(A)\}$ be the set of
central primitive idempotents of $A$. Then it is easily seen that
the set $P'=\{\varepsilon_{\chi^f}\mid\chi \in \Irr(A)\}$ is the
set of central primitive idempotents of $A'$, where $\chi^f(a') =
\chi(f^{-1}(a'))$ and $a'\in A'$. It follows that for any $\chi\in
\Irr(A)$ there exists $\psi \in \Irr(A)$ such that
$(\varepsilon_\psi)^f = \varepsilon_{\chi^f }$, and so $\psi(1) =
\chi(1)$. Therefore, by comparing the coefficient of $1_{A'}$ in
the both sides of the former equality we get
$$
\frac{\psi(1)}{|B^+|}\zeta_\psi  =
\frac{\chi(1)}{|B'^+|}\zeta'_{\chi^f}
$$
where $\zeta_{\psi}$ and $\zeta'_{\chi^f}$ are the standard feasible
multiplicities of $\psi$ and $\chi^f$ with respect to standard feasible
traces $\zeta$ and $\zeta'$, respectively. This implies that
$\zeta_\psi=\zeta'_{\chi^f}$. Therefore the set of standard
feasible multiplicities of the C-algebras $(A,B)$ and $(A',B')$ are
the same, as desired.\hfill\bull
\end{proof}

\section{The standard character }\label{stchar}

Let $X$ be a set with $n$ elements.
According to \cite{WI} a linear subspace $W$ of the algebra
$\Mat_n(\mathbb{C})$, the set of all $n\times n$
matrices with entries in $\mathbb{C}$,
is called a {\it cellular algebra}
on $X$ if $I_n, J_n\in W$; $W$ is closed under the matrix and the Hadamard (componentwise)
multiplications and $W$ is closed under transpose, where $I_n$ is the identity matrix
and $J_n$ is the matrix all of whose entries are ones. For example,
any complex adjacency algebra of an association scheme is a cellular
algebra. Conversely, in the sense of \cite{HIG} for a given cellular
algebra $W$ on a finite set $X$,
there is a coherent configuration on $X$ whose adjacency algebra is $W$.
So cellular algebras and adjacency algebras are equivalent objects, see \cite{WI}.
On the other hand, any cellular algebra is a table algebra but the converse is
not true, see Example {\ref{06100701}}. In this section
we are interested in finding a necessary
and sufficient condition for which a commutative table
algebra becomes a cellular algebra.

Let $(X,G)$ be an association scheme and let $\C G =
\bigoplus_{g\in G} \C\sigma_g$ be the complex adjacency algebra of
$G$. Then the representation of $\C G$ which sends $\sigma_g$ to
itself for every $g\in G$ affords a character $\gamma_G$ which is
called {\it standard character} of $\C G$, see \cite{Zi}.
Moreover, $\gamma_G(\sigma_{1_X}) = |X|$ and $\gamma_G(\sigma_g) =
0$ for $1_X\neq g \in G$ and
\begin{eqnarray}\label{081107st}
\gamma_G = \displaystyle\sum_{\chi \in \Irr(G)}m_{\chi}\chi.
\end{eqnarray}
In this case by setting $A = \C G$ and $B = \{\sigma_g: g\in G\}$,
the pair $(A,B)$ is a C-algebra with the standard feasible trace $\zeta=\gamma_G$
given in (\ref{081107st}). Therefore, the standard feasible
multiplicities $\zeta_\chi = m_{\chi}$ for $\chi \in \Irr(G)$ are
nonnegative integers. \vspace{0.4cm}

In general, we do not know if $\zeta_\chi$'s are nonnegative
integers, or equivalently wether or not $\zeta$ is a character. It
is interesting to find some examples of C-algebras apart from
association schemes, for which $\zeta$ is a character. In example
below we give a commutative table algebra which does not come from
association schemes and for it $\zeta$ is a character. In the case
that the standard feasible trace $\zeta$ of a C-algebra $(A,B)$ is
a character,  by pattern of the theory of
association schemes we call $\zeta$ the {\it standard character} of $(A,B)$.\\
\begin{definition}
We say that a C-algebra has standard character condition, if it
possesses the standard character. We denote by $\S$ the class of
all such C-algebras.
\end{definition}
Clearly  association schemes belong to the class $\S$ and
Example \ref{06100701} below shows that the class $\mathcal{S}$ is
larger than the class of association schemes. But this class is not equal
to the class of integral table algebras, in fact in Example \ref{examnonstand} below
we give an integral table algebra does not belong to $\S$.

For a given strongly regular graph $(X,E)$ with parameters
$(n,k,\lambda,\mu)$ one can find an association scheme $\CC=(X,G)$
where $G = \{1_X,g,h\}$ with structure constants $\lambda_{gg{1_X}}
= k, \lambda_{ggg} = \lambda, \lambda_{ggh} = \mu$. In \cite{Bro}
some of the necessary conditions for the existence of a strongly
regular graph with parameters  $(n,k,\lambda,\mu)$ are given. One
of them is {\it integrality condition}. If we consider adjacency
algebra of association scheme $\CC$, which is a C-algebra $(A,B)$
of dimension 3, then one can see that the standard character condition for
$(A,B)$ is equivalent to integrality condition for the
existence strongly regular graphs with parameters
$(n,k,\lambda,\mu)$, see \cite{Bro}.
\\

In Example \ref{06100701} we will use the definition of a {\it
finite affine plane} in the sense of \cite{BJL}. We recall that
for every finite affine plane $\mathcal{P} = (P,L,I)$,
i.e., $P\cup L$ is finite, there exists an integers $q\geq 2$,
called the order of affine plane such that $|P|=q^2$ and $|L|=q^2+q$,
and each line is incident to exactly $q + 1$ points. Besides,
there are exactly $q+1$ classes $B_1,\ldots ,B_{q+1}$ of pairwise
parallel (nonintersecting) lines, each of which
is of cardinality $q$. \\
\begin{example}
{\rm (cfg. \cite{PR})\label{06100701}
Let $\mathcal{P}$ be a finite affine plane with point set $P$ and line set $L$. Let $\Mat_P(\C)$ be the algebra
of all $|P|\times|P|$ complex matrices whose rows and columns are
indexed by the elements of $P$. Define a $(0,1)$-matrix $r_i\in
\Mat_P(\C)$ by
\[ (r_i)_{u,v} = \left\{ \begin{array}{ll} 1, & \text{if }u\neq v ~~\text{and}~~ l(u,v)\in B_i
\\ 0, & \text{otherwise} \end{array} \right.\]
where $l(u,v)$ is the line incident with both $u$ and $v$. Then
for all $i,j=1,\ldots ,q+1$ we have
\[r_ir_j=\left\{\begin{array}{ll} (q-1)r_0+(q-2)r_i, & \text{if}~~~ i=j\\
\displaystyle\sum_{k\neq {0,i,j}}r_k, & \text{if}~~~{i\neq
j}\end{array}\right.\] where $r_0$ is the identity matrix. So the
set $B=\{r_0,\ldots,r_{q+1}\}$ is a linear base of the subalgebra
A of the algebra $\Mat_P(\C)$ generated by $B$. Then it is easily
seen that $(A,B)$ is a table algebra (with $*$ being the Hermitian
conjugation in $\Mat_P(\C)$). An easy computation shows that the
character table of the table algebra $(A,B)$ is as follows:
$$\begin{array}{c|ccccccc|c}
       & r_0 &r_1  & r_2 & .& .& .&r_{q+1} &\zeta_{\chi_i}\\
     \hline \chi_1 & 1 &q-1  &q-1  & .& .& .& q-1&1 \\
      \chi_2& 1  &q-1 &-1 &.  &. &. &-1&q-1  \\
      \chi_3     & 1 &-1  &q-1  &.  &.  & .&-1 &q-1 \\
                 .& .& .& .& .& .& .& .&.\\
                .& .& .&. &. &. & .& .&.\\
                .& .& .&. & .&. & .& .&.\\
       \chi_{q+2}& 1 &-1  &-1  &  .&  .& .& q-1& q-1\\
    \end{array}
$$
\begin{center}
\text{Table~ (1)}
\end{center}
From Table (1) one can see that the standard feasible
multiplicities of the characters of the table algebra $(A,B)$ are
positive integers. Thus $\zeta $ is a character.  Now we claim
that by a suitable integer $q$ the table algebra $(A,B)$ is not
the complex adjacency algebra of any association scheme. To do so,
suppose on the contrary that the table algebra $(A,B)$ is the
complex adjacency algebra of an association scheme $(X,G)$, where
$G = \{g_0,\ldots, g_{q+1}\}$. Then for each $i, 1\leq i \leq
q+1$, the subset $\{g_0,g_i\}$ of $G$ is a closed subset and so
$E_i = g_0\cup g_i$ is an equivalence relation on $X$. Now let $L$
be the set of all equivalence classes of the equivalence relations
$E_i, 1\leq i \leq q+1$. Then it is easily seen that the sets $X$
and $L$ form an affine plane of order $q$ consisting $X$ as the
set of points and $L$ as the set of lines. But we can choose a
suitable integer $q$ in such a way that there is no affine plane of
degree $q$ (see \cite{BJL}), we get a contradiction. Thus $(A,B)$
can not come from an association scheme.}
\end{example}
\begin{example}\label{examnonstand}{\rm
Let $A$ be a $\C$-linear space with the basis $B=\{1_A,b,c\}$ such
that
\begin{eqnarray}
b^2 &=&2~1_A+ b\nonumber \\
c^2&=&25~1_A+25 b+22c \nonumber \\
bc&=&cb=2 c \nonumber
\end{eqnarray}
Then one can see that the pair $(A,B)$ is an integral table algebra.
By using the orthogonality relation given in Lemma \ref{130308} part (ii)
the character table of $(A,B)$ is as the following:

$$\begin{array}{c|ccc|c}
       & 1_A &b  &c  &\zeta_{\chi_i}\\
     \hline \chi_1 & 1 &2  &25  &1 \\
      \chi_2& 1 &2 &-3 &\frac{25}{3}  \\
      \chi_3& 1 &-1 &0 & \frac{56}{3}\\
    \end{array}
$$
\begin{center}
Table (2)
\end{center}
Thus from Table (2) the standard feasible multiplicities of the characters of
$(A,B)$ are not integers. This shows that $(A,B)\notin \S$.}
\end{example}

\begin{lemma}\label{embed}
Let $(A,B)\in \S$ be a commutative C-algebra. Then any matrix
representation $D$ of $A$ which affords $\zeta$ is faithful.
\end{lemma}
\begin{proof}
Let $D$ be a matrix representation of $A$ which affords
$\zeta$. Suppose that $a = \sum_{b_i\in B}x_i b_i \in A$ is in the kernel of $D$, so
$D(a) = 0$. Since $A$ is commutative, there is a non-singular
matrix $P$ such that for all $b_i\in B$ the following equality
holds:
$$
PD(b_i)P^{-1} = {\rm diag}(\chi_1(b_i),\underbrace{\chi_2(b_i), \ldots
,\chi_2(b_i)}_{\zeta_{\chi_2}\rm{-times}}, \ldots, \underbrace
{\chi_n(b_i), \ldots ,\chi_n(b_i)}_{\zeta_{\chi_n}\rm{-times}})
$$
where $\Irr(A) = \{\chi_1,\ldots,\chi_n \}$. It follows that
$\sum_{b_i\in B} x_iPD(b_i)P^{-1} =0$. Therefore, $MX = 0$, where $M$
is an $n\times n$  matrix whose $(i,j)$ entry is $\zeta_{\chi_i}\chi_i(b_j)$ and
$X$ is a column matrix whose $i$-th entry is $x_i$. Now since
$M$ is a non-singular matrix, it follows that $X = 0$ which implies that
$a = 0$. This completes the proof of the lemma. \hfill \bull
\end{proof}
\vspace{.4cm}

\begin{remark}
Let $(A,B)\in \S$ be a table algebra and let $D$ be a matrix
representation of $A$ which affords the standard character
$\zeta$. Then one can check that $(D(A),D(B))$ is a table algebra
by defining $D(b)^*=D(b^*)$ and $|D(b)|=|b|$.
\end{remark}
\vspace{.4cm}
We say that a table algebra $(A,B)$ has an adjacency algebra
homomorphic image, if there are an association scheme $(X,G)$ and
a C-algebra epimorphism $T: (A,B) \twoheadrightarrow (\mathbb{C}G,
C)$, where $C = \{\sigma_g : g \in G\}$ is the basis of the
adjacency algebra $\mathbb{C}G$.
\begin{theorem}\label{cel}
Let $(A,B)$ be a table algebra. Then $(A,B)$ has an adjacency
algebra homomorphic image iff $(A,B)\in \S$ and a matrix
representation $D$ which affords $\zeta$ satisfies the following
conditions for any $b\in B$:
\nmrt
\tm{1} $D(b^*)=D(b)^t$.
\tm{2} $D(b)$ is a
$(0,1)$-matrix. \enmrt
\end{theorem}
\begin{proof}
We first prove the necessity of conditions (1) and (2). Let
$\mathbb{C}G$ be an adjacency algebra which is a homomorphic image
of $(A,B)$. So there exists a C-algebra homomorphism $T$ from $A$
onto $\mathbb{C}G$. Then $T(A)=\mathbb{C}G$ and $T(b^*)=T(b)^t$.
It follows that $|b|=|T(b)|$, for $b\in B$. So $T$ induces a
matrix representation $D$ of degree $|B^+|$ and conditions (1) and
(2) are valid for $D$. Then the character $\chi$ which is
afforded by $D$ has values $|B^+|$ at $1_A$ and 0 at any $b\in
B\setminus \{1_A\}$. This implies that $\chi$ is the standard
character $\zeta$ of $(A,B)$ and so $(A,B)\in \S$, as desired.

Conversely, suppose that $(A,B)\in \S$
and conditions (1) and (2) hold for a
matrix representation $D$ of $A$
which affords the standard character $\zeta$.
Since
$$D(b)D(b)^t = D(b)D(b^*) = |b|D(1_A)+\sum_{1_A\neq
b\in B}\lambda_{bb^*d}D(d)
$$
we conclude that the matrix $D(b)$ contains $|b|$ ones in each
rows and columns. On the other hand, let $b\in B$ such that
$D(b)_{ij}=1$. If $D(c)_{ij}=1$ for some $c\in B$, then the
$(i,i)$ entry of matrix $D(b)D(c)^t=1$. It follows that $b=c$.
Thus the matrices $D(b), b\in B$ are disjoint and the sum of them
is the matrix $J_n$ where $n=|B^+|$. This implies that
$(D(A),D(B))$ is a cellular algebra ( or an adjacency algebra),
as desired.
\hfill\bull
\end{proof}
\begin{corollary}
Let $(A,B)$ be a commutative table algebra. Then $(A,B)$ comes
from an association scheme iff $(A,B)\in\S$ and a matrix
representation $D$ which affords
$\zeta$ satisfies the following conditions for any $b\in B$:
\nmrt
\tm{1}
$D(b^*)=D(b)^t$.
\tm{2} $D(b)$ is a $(0,1)$-matrix.
\enmrt
\end{corollary}
\begin{proof}
This is an immediate consequence from Theorem \ref{cel} and Lemma
\ref{embed}. \hfill\bull
\end{proof}
\vspace{.4cm}
Let $(A,B)$ be a C-algebra. The coordinate-wise multiplication
$\circ$ with respect to the basis $B$ by $b\circ c=\delta_{bc}b$,
for $b,c\in B$ is defined in the sense of \cite{EPV}. We say that
a matrix representation $D$ of $A$ preserves Hadamard products if
$D(b\circ c)=D(b)\circ D(c)$, for $b,c\in B$.

For a matrix $A$, $\tau(A)$ denotes the sum of all entries $A$.
One can see that for any two square matrices $A$ and $B$ of the
same size:
$$
\tau(A\circ B) = {\rm tr}(A B^t) = {\rm tr}(A^t B).
$$

\begin{corollary}\label{hadam}
Let $(A,B) \in \S$ be a table algebra and
let $D$ be a matrix representation of $A$ which affords $\zeta$.
Then table algebra the
$(D(A),D(B))$ is a cellular algebra iff $D$ perseveres
Hadamard products.
\end{corollary}
\begin{proof}
The necessity is obvious. For the sufficiency, since $D(b), b\in
B$ persevere Hadamard products, each $D(b), b\in B$ is $(0,1)$-
matrix. On the other hand,
$$
\tau(D(b^*)\circ D(c)^t)={\rm
tr}(D(b^*)D(c)) ~~  b,c\in B.
$$
But ${\rm tr}(D(b^*)D(c))=0$ iff
$b\neq c$. Thus $D(b^*)=D(b)^t$. Now the result follows from
Theorem \ref{cel} and we are done. \hfill\bull
\end{proof}

\vspace{.4cm} In the rest of this section, we suppose that $(A,B)$
is a commutative C-algebra of dimension $d$ with the set of the
primitive idempotents $\{\varepsilon_{\chi} |~ \chi \in \Irr(A)
\}$. Then from \cite[Section 2.5]{Ban} there are two matrices $P =
(p_b(\chi))$ and $Q=(q_{\chi}(b))$ in $\Mat_d(\mathbb{C})$, where
$b\in B$ and $\chi \in \Irr(A)$ such that $PQ = QP = |B^+|I$ and
\begin{eqnarray}\label{25100701}
b = \displaystyle\sum_{\chi\in \Irr(A)}
p_b(\chi)\varepsilon_{\chi}~~~ \text{ and } ~~ \varepsilon_{\chi}
= \frac{1}{|B^+|}\displaystyle\sum_ {b\in B} q_{\chi}(b)b.
\end{eqnarray}
Then from Remark (\ref{251007r}) and (\ref{25100701}) we get
\begin{eqnarray}\label{25100703}
q_{\chi}(1_A) = \zeta_{\chi} ~~~\text{and}~~~ \chi(b) = p_b(\chi),
\end{eqnarray}
where $b\in B$ and $\chi\in \Irr(A)$. The dual of $(A,B)$ in the
sense of \cite{Ban} is as follows: with each linear representation
$\Delta_{\chi} : b\mapsto p_b(\chi)$, we associate the linear
mapping $\Delta^*_{\chi}: b \mapsto q_{\chi}(b)$. Since
$Q=(q_{\chi}(b))$ is non-singular, the set $\widehat{B} =
\{\Delta^*_{\chi}:  \chi \in \Irr(A)\}$ is a linearly independent
and so form a base of the set of all linear mapping $\widehat{A}$
of $A$ into $\C$. From \cite[Thorem 5.9]{Ban} the pair
$(\widehat{A},\widehat{B})$ is a C-algebra with the identity
$1_{\widehat{A}} = \Delta ^*_{\rho}$ and involutory automorphism
which maps $\Delta^*_{\chi}$ to $\Delta^*_{\overline{\chi}}$,
where $\overline{\chi}$ is complex conjugate to $\chi$. The
C-algebra $(\widehat{A},\widehat{B})$ is called the {\it dual
C-algebra} of $(A,B)$. Moreover, the structure constants of
$(\widehat{A},\widehat{B})$ which are given in \cite[(5.26)]{Ban}
can be written as the following
\begin{eqnarray}\label{25100702}
q_{\varphi \psi}^{\chi} =
\frac{\zeta_{\varphi}\zeta_{\psi}}{|B^+|}\displaystyle\sum_{b\in
B} \frac{1}{|b|^2}p_b(\varphi)p_b(\psi)\overline{p_b(\chi)}
\end{eqnarray}
which are real numbers, where $\overline{p_b(\chi)}$ is the
complex conjugate to $p_b(\chi)$. From (\ref{25100702}) and
(\ref{23100702}) one can see that $q_{\chi,\overline{\chi}}^{\rho}
= \zeta_{\chi}$. Then $|\widehat{B}^+| = \sum_{\chi \in
\Irr(A)}\zeta_{\chi}$. The primitive idempotents $f_b, b\in B$ of
$\widehat{A}$ are given by \cite[5.23]{Ban} as the following
\begin{eqnarray}\label{idmdual}
f_b = \frac{1}{|\widehat{B}^+|}\displaystyle\sum_{\chi \in
\Irr(A)}p_b(\chi)\Delta^*_{\chi}.
\end{eqnarray}
\begin{lemma}\label{181107}
Keeping the notation above, there is a bijection correspondence
between the standard feasible multiplicities of the characters of
$(\widehat{A},\widehat{B})$ and the degrees of basis elements $B$.\hfill\bull
\end{lemma}
\begin{proof}
 From (\ref{idmdual}), one can see that the coefficient of the unit
element $1_{\widehat{A}}$ of $\widehat{A}$ in the linear
decomposition of $|\widehat{B}^+| f_b$ in terms of the basis
elements $\widehat{B}$ is equal to $p_b(\rho)$. On the other hand,
from the equation of the right hand side of (\ref{25100703}) we
get $p_b(\rho)=\rho(b) = |b|$. But from Remark \ref{251007r} any
standard feasible multiplicity of the characters of $(\widehat{A},\widehat{B})$ corresponds to
the number $p_b(\rho)$ for some $b\in B$, as desired. \hfill \bull
\end{proof}
\vspace{.4 cm}

A C-algebra is called {\it integral degree} if its all degrees
$|b|, b\in B$, are integer.
\begin{corollary}
Let $(A,B)$ be a C-algebra. Then $(A,B)$ is integral
degree and belongs to $\S$ iff so is
$(\widehat{A},\widehat{B})$.
\end{corollary}

\begin{proof}
Let $(A,B)$ be a C-algebra and $(\widehat{A},\widehat{B})$ be its
dual with the standard feasible traces $\zeta$ and
$\widehat{\zeta}$, respectively. To prove the necessity, since
$(A,B)$ is in $\mathcal{S}$ the equality
$q^{\rho}_{\chi,\overline{\chi}} = \zeta_{\chi}$ implies that
$(\widehat{A},\widehat{B})$ is integral degree. Since $(A,B)$ is
integral degree, from Lemma \ref{181107} we conclude that
$(\widehat{A},\widehat{B})$ is in $\mathcal{S}$.

To prove the sufficiency, by the necessity we see that
$(\widehat{\widehat{A}}, \widehat{\widehat{B}})\in \mathcal{S}$ is
integral degree. Now the proof follows from Lemma \ref{251107} and
the Duality Theorem \cite[Theorem 5.10]{Ban}, i.e., $(A,B) \simeq
(\widehat{\widehat{A}}, \widehat{\widehat{B}})$. \hfill \bull
\end{proof}

\end{document}